\documentclass{amsart}
\usepackage{amsmath,amssymb,amsthm,amscd, mathtools, latexsym}
\usepackage{enumerate,varioref, dsfont}
\usepackage{tikz-cd}
\usepackage{enumitem}

\usepackage{hyperref}
\usepackage{url}

\newtheorem{Thm}{Theorem}[section]

\newtheorem{Prop}{Proposition}[section]

\newtheorem{Conj}{Conjecture}[section]
\newtheorem{Rem}{Remark}[section]

\newtheorem{Def}{Definition}[section]

\newenvironment{claim}[1]{\par\noindent\textbf{Claim:}\space#1}{}

\newcommand{\Sym}{\text{Sym}}

\def\cit{{\mathbb C}}

\def\qit{{\mathbb Q}}
\def\zit{{\mathbb Z}}

\def\pit{{\mathbb P}}

\def\0{{\mathcal O}}

\def\Hom{\mathop{\rm Hom}\nolimits}

\def\End{\mathop{\rm End}\nolimits}

\def\h{{\mathfrak h}}

\def\h{{\mathfrak h}}
\def\S{{\mathfrak S}}

\def\M{{\mathcal M}}

\def\V{{\mathcal V}}

\def\Q{{\mathbb{Q}}}

\begin{document}

\title{THE MOTIVE OF THE FANO SURFACE OF LINES}

\author{Humberto Diaz}
\email{hdiaz123@math.duke.edu}
\address{DEPARTMENT OF MATHEMATICS, DUKE UNIVERSITY, DURHAM, NC}

\thanks{The author would like to thank his advisor Prof. C. Schoen for his patience in reading carefully several early drafts of this and for his helpful comments. He would also like to thank Prof. B. Kahn for his enthusiasm in reading a later draft and for his help in improving the first argument. }

\begin{abstract}
In this short note, we prove that the Chow motive of the Fano surface of lines on the smooth cubic threefold is finite-dimensional in the sense of Kimura. This gives an example of a variety not dominated by a product of curves whose Chow motive is of Abelian type. 
\end{abstract}

\date{}

\maketitle

\section{Introduction}
Let $k$ be a field with $\text{char } k \neq 2$ and let $X \subset \pit^{4}$ be a smooth cubic threefold. We denote by $S(X)$ (or just $S$) the Fano variety of lines in $X$. This is known to be a smooth, connected projective surface of general type. It turns out that this surface possesses a great many remarkable properties. It is known, for instance, that the Albanese map is an imbedding $i: S \hookrightarrow A:= Alb(S)$ and that the pull-back $H^{2} (A) \xrightarrow{i^{*}} H^{2} (S)$ is an isomorphism. In this note, we will prove a motivic version of this isomorphism. For this, let $\M_k$ denote the category of pure Chow motives {with rational coefficients} and let $\h: \V_{k}^{opp} \mapsto \M_{k}$ be the functor that sends a smooth projective $k$-variety $X$ to $\h(X) = (X, \Delta_{X}, 0) \in \M_{k}$. Then, our first result is that:
\begin{Thm}\label{1} The morphism $\h(i) : \h (A) \xrightarrow{} \h (S)$ is split-surjective. \end{Thm}
This will show that the motive of the Fano surface is {\em finite-dimensional in the sense of Kimura}. The primary known examples of surfaces with finite-dimensional motive are those for which either: \begin{enumerate}[label=(\roman*)] \item\label{rep} The Chow group of nullhomologous $0$-cycles is representable. \item\label{dom-prod} There exists a dominant rational map from a product of smooth projective curves. \end{enumerate} Since $p_{g} (S) > 0$, Mumford's Theorem (\cite{V} Theorem 3.13) shows that $CH^{2}_{hom} (S)$ is not representable (when $k = \cit$). Moreover, a recent result in \cite{S} shows that $S$ is not dominated by a product of curves. To the author's knowledge, the Fano surface is the first example of a surface with finite-dimensional motive for which neither \ref{rep} nor \ref{dom-prod} holds. One reason for interest in finite-dimensionality is that if $M \in \M_{k}$ is finite-dimensional, then any morphism $f: M \to M$ that induces an isomorphism on cohomology is actually an isomorphism of motives. This will be important to proving:
\begin{Thm}\label{2} The pull-back along the Albanese imbedding $i: S \hookrightarrow A$ induces an isomorphism $\h^{2} (A) \xrightarrow[]{\cong} \h^{2} (S)$ in $\M_{k}$. \end{Thm}

The plan will be as follows. We will give a proof of the Theorem \ref{1}, then review the notion of finite-dimensionality for motives which will allow us to prove Theorem \ref{2}. In this note, all Chow groups will be taken with {rational coefficients}.

\section{A General Principle}

The following variant of the Manin principle will facilitate the proof of Theorem \ref{1}, as pointed out by B. Kahn.

\begin{Prop}\label{BK} Let $X$ and $Y$ be smooth projective connected varieties over a field $k$ with a morphism $f: Y \to X$ such that $f_{\Omega}^{*} : CH^{i} (X_{\Omega}) \to CH^{i} (Y_{\Omega})$ is surjective for all $i$ and every algebraically closed extension $k \subset \Omega$. Then, $\h(f) : \h(X) \vdash \h(Y)$ is split-surjective.
\end{Prop}

\noindent This principle has appeared in the literature in various guises (see, for instance, \cite{J} Theorems 3.5 and 3.6). However, since this exact statement is difficult to find, we give a proof below, following \cite{KA}. We subdivide the proof into three claims:
\vspace{2 mm}
\begin{claim} $f_{L}^{*} : CH^{i} (X_{L}) \to CH^{i} (Y_{L})$ is surjective for all $i$ and every extension $k \subset L$.
\begin{proof}[Proof of Claim] Let $\beta \in CH^{i} (Y_L)$. We need to find $\alpha \in CH^{i} (X_{L})$ for which $f_{L}^{*}(\alpha) = \beta$. By assumption, $f_{\overline{L}}^{*} : CH^{i} (X_{\overline{L}}) \to CH^{i} (Y_{\overline{L}})$ is surjective, so there is some $\overline{\alpha} \in CH^{i} (X_{\overline{L}})$ such that $f_{\overline{L}}^{*}(\overline{\alpha}) = \beta_{\overline{L}}$. Let $L' \supset L$ be some finite Galois extension for which there is $\alpha' \in CH^{i} (X_{L'})$ with $\alpha'_{\overline{L}} = \overline{\alpha}$. Indeed, there is some finite extension for which this holds; since Chow groups are unchanged by passing to a purely inseparable extension, we can assume that $L'$ is a separable extension. This allows us to pass to a Galois extension. Then, let $G = Gal(L'/L)$ and define $\displaystyle \text{Tr}(\alpha') = \frac{1}{|G|} \sum_{g \in G} g^{*}\alpha'$. By \cite{F} Example 1.7.2.6, this yields $\alpha \in CH^{i} (X_{L})$ such that $\alpha_{L'} = \text{Tr}(\alpha')$. We compute:
$$(f_{L}^{*}(\alpha))_{L'} = f_{L'}^{*}(\alpha_{L'}) = f_{L'}^{*}(\text{Tr}(\alpha')) = \text{Tr}(f_{L'}^{*}(\alpha')) = \text{Tr}(\beta_{L'}) = \beta_{L'}$$
Since $CH^{i} (X_{L}) \to CH^{i} (X_{L'})$ is injective, it follows that $f_{L}^{*}(\alpha) = \beta$, as desired.
\end{proof} 
\end{claim}

\begin{claim} $(f \times id_{Z})^{*} : CH^{i} (X \times Z) \to CH^{i} (Y \times Z)$ is surjective for all $i$ and all projective varieties $Z$ over $k$.
\begin{proof}[Proof of Claim] One easily reduces to the case that $Z$ is connected. We first observe that the pull-back $$(f \times id_{Z})^{*} : CH^{i} (X \times Z) \to CH^{i} (Y \times Z)$$ is defined even when $Z$ is singular. Indeed, $f$ is a local complete intersection morphism since the source and the target are smooth, and one readily checks that so is the (fiber) product $$f \times id_{Z}: X \times Z \to Y \times Z.$$ Thus, using \cite{F} Chapter 6.6, we can define $(f \times id_{Z})^{*}.$ We then observe the following commutative diagram with rows exact:   
\begin{equation}\begin{CD}\displaystyle
\lim_{\xrightarrow[W]{}} CH^{i-c} (X \times W) @>{({id_{X} \times j_{W})_{*}}}>> CH^{i} (X \times Z) @>>> CH^{i} (X_{k(Z)}) @>>> 0\\
 @V{(f \times id_{W})^{*}}VV  @V{(f \times id_{Z})^{*}}VV  @V{(f_{k(Z)})^{*}}VV \\
\displaystyle \lim_{\xrightarrow[W]{}} CH^{i-c} (Y \times W) @>{({id_{Y} \times j_{W})_{*}}}>> CH^{i} (Y \times Z) @>>> CH^{i} (Y_{k(Z)}) @>>> 0
\label{main}
 \end{CD}\end{equation}
 where $k(Z)$ is the function field of $Z$, the limit ranges over subvarieties $j_{W}: W \hookrightarrow Z$ with $dim (W) \leq dim (Z)$ and $c = dim (Z) - dim(W)$. The localization sequence implies that the rows are exact. For the commutativity of the left diagram, note first that there is a Cartesian diagram:
 \begin{equation}\begin{CD}
 Y \times W @>{id_{Y} \times j_{W}}>> Y \times Z\\
 @V{f \times id_{W}}VV @V{f \times id_{Z}}VV\\
  X \times W @>{id_{X} \times j_{W}}>> X \times Z
 \end{CD}\end{equation}
 with the left and right vertical arrows of the same relative codimension. Then, using \cite{F} Theorem 6.2, one obtains the desired commutativity.\\
 \indent The claim then follows by an induction argument on $n=dim(Z)$. The case $n=0$ follows from the previous claim. Assume then that it holds for $n-1$. Then, we note that the rightmost vertical arrow in (\ref{main}) is surjective by the previous claim. The leftmost arrow is surjective by the inductive hypothesis. A diagram chase then shows that the statement is true for $n$. Hence, the claim.  
 \end{proof}
\end{claim}
\begin{claim} $\h(f)$ possesses a right-inverse. 
\begin{proof}[Proof of Claim] Taking $Z = Y$ and $i = dim(Y)$ in the above claim, we obtain that $$CH^{i} (Y \times X) \xrightarrow{(id_{Y} \times f)^{*}} CH^{i} (Y \times Y)$$ is surjective. So, there is some $\gamma \in CH^{i} (X \times Y)$ for which $(id_{Y} \times f)^{*}\gamma = \Delta_{Y}$. Applying Liebermann's lemma (\cite{F} Proposition 16.1.1) then gives $$\h(f)\circ \gamma = (id_{Y} \times f)^{*}\gamma = \Delta.$$ This is the desired result.
\end{proof}\end{claim}

\section{Proof of Theorem \ref{1}}

Assume that $S$ has a $k$-rational point and let $i: S \to A$ be the corresponding Albanese morphism. To prove the theorem, we need to show that the correspondence $\prescript{t}{}{\Gamma_{i}} \in Cor^{0} (A \times S)$ possesses a right-inverse.  The general principle then shows that it suffices to prove that the pull-back $i_{\Omega}^{*} : CH^{j} (A_{\Omega}) \to CH^{j} (S_{\Omega})$ is surjective for $j = 1, 2$ and all algebraically closed extensions $k \subset \Omega$. \\
 \indent For a smooth cubic threefold $X$ over $k$, $S$ is the Hilbert scheme of $X$ with Hilbert polynomial $t+1$. We recall the following base change compatibility for Hilbert schemes (see, for instance, \cite{N}); i.e., for an extension $k \subset \Omega$, we have
 $$S(X_{\Omega}) \cong S(X)_{\Omega}$$ 
Thus, $S_{\Omega}$ is the Fano surface of lines of $X_{\Omega}$. Moreover, since the Albanese is compatible with base extension, we view $i_{\Omega}: S_{\Omega} \to A_{\Omega}$ as an Albanese map.

\begin{Prop}\label{begin} For all algebraically closed extensions $k \subset \Omega$, the pull-back $i_{\Omega}^*: CH^{1} (A_{\Omega}) \to CH^{1} (S_{\Omega})$ is an isomorphism.
\begin{proof} Since $i_{\Omega}: S_{\Omega} \to A_{\Omega}$ is an Albanese map, we have $Pic^{0} (A_{\Omega}) \xrightarrow{i_{\Omega}^*} Pic^{0} (S_{\Omega})$ is an isomorphism, and so it suffices to show that we have an isomorphism $i_{\Omega}^*: NS (A_{\Omega})_{\qit} \to NS (S_{\Omega})_{\qit}$ of N{\'e}ron-Severi groups. Now, let $\ell \neq \text{char } k$. Then, $H^{2} (A_{\Omega}, \qit_{\ell} (1)) \xrightarrow{i_{\Omega}^*} H^{2} (S_{\Omega}, \qit_{\ell} (1))$ is an isomorphism by \cite{R} Proposition 4. Note that $S$ and $A$ may be defined over some finitely generated field $k_{0}$. Upon passing to a large enough extension of ${k_{0}}$, we may also assume that $C$ possesses a model over ${k_{0}}$ and that $NS(C_{k_{0}} \times A_{k_{0}})_{\qit} \cong NS(C \times A)_{\qit}$ (and, similarly, for $C \times S$). Note that this is possible because the N{\'e}ron-Severi group is finitely generated. Thus, we need to show that $$NS(C_{k_{0}} \times A_{k_{0}})_{\qit} \xrightarrow[\cong]{(id_{C} \times i)^*} NS(C_{k_{0}} \times S_{k_{0}})_{\qit}$$  For this, let $G := Gal(k/k_{0})$ be the absolute Galois group and $\ell \neq \text{char } k$. The K{\"u}nneth theorem on cohomology then gives a ($G$-module) isomorphism:
\begin{equation} H^{2} (C \times A, \Q_{\ell}(1)) \xrightarrow[\cong]{(id_{C} \times i )^*} H^{2} (C \times S, \Q_{\ell}(1)) \label{H^{2}} \end{equation} The result then follows from applying the functor $H^{0} (G, -)$ to (\ref{H^{2}}) and noting that the Tate conjecture holds for the left-hand side (by Faltings' theorem and the fact that $k_{0}$ is a finitely generated field).  \end{proof} 
\end{Prop}
We have the following important result of Bloch:
\begin{Prop}[(\cite{BL} 1.7)]\label{Bloch} Let $S$ be the Fano surface of lines of a smooth cubic threefold in $\pit^4$ over an algebraically closed field of characteristic $\neq 2$. Then, the intersection product $CH^{1} (S) \otimes CH^{1} (S) \xrightarrow{\cdot} CH^{2} (S)$ is surjective.
\end{Prop}
\begin{Prop}\label{Bloch-lemma} For all algebraically closed extensions $k \subset \Omega$, the pullback $CH^{2} (A_{\Omega}) \xrightarrow{i_{\Omega}^{*}} CH^{2} (S_{\Omega})$ is surjective.
\begin{proof} Since pull-back commutes with intersection product on Chow groups, we have the following diagram:
\begin{equation}
\begin{CD}
CH^{1} (A_{\Omega}) \otimes CH^{1} (A_{\Omega}) @>\cdot>> CH^{2} (A_{\Omega})  \\
@Vi_{\Omega}^{*} \times i_{\Omega}^{*} VV @Vi_{\Omega}^{*}VV   \\
CH^{1} (S_{\Omega}) \otimes CH^{1} (S_{\Omega})  @>\cdot>> CH^{2} (S_{\Omega}) 
\end{CD}
\end{equation}
From Proposition \ref{begin}, the left vertical arrow is surjective. Since $\Omega$ is algebraically closed and $S_{\Omega}$ is the Fano surface of lines of $X_{\Omega}$, we can apply Proposition \ref{Bloch} to deduce that the bottom horizontal arrow is also surjective.. So, the right vertical arrow is surjective, as desired.
\end{proof} 
\end{Prop}

\section{Finite-dimensionality of Motives}

This section reviews the definition and properties of finite-dimensionality before proving Theorem \ref{2}.
Recall that the category of Chow motives $\M_{k}$ is a tensor category with tensor product defined as:
\begin{equation}
(X, \pi, m)  \otimes  (Y, \tau, n)  :=  (X \times Y, \pi \times \tau, m+n)
\end{equation}
We can define an action of the symmetric group $\Q[\S_{n}] \to End_{\M_{k}} (M^{\otimes n})$ for $M \in \M_{k}$. Since $\M_{k}$ is a pseudo-Abelian category, all idempotents possess images in $\M_{k}$. So, for any idempotent in the group algebra $\Q[\S_{n}]$, there is a corresponding motive. In particular, we have \begin{equation}\begin{split} \Sym^{n} M = \text{Im}(\pi_{sym})\\ \wedge^{n} M = \text{Im}(\pi_{alt})\end{split}\end{equation}
for the symmetric and the alternating representation of $\S_{n}$. 
\begin{Def}[(Kimura)] A motive $M \in M_{k}$ is said to be oddly finite-dimensional if $\Sym^{n} M = 0$ for $n >> 0$ and evenly finite-dimensional if $\wedge^{n} M = 0$ for $n >> 0$. $M$ is said to be finite-dimensional if $M = M_{+} \oplus M_{-}$, where $M_{+}$ is evenly finite-dimensional and $M_{-}$ is oddly finite-dimensional. \end{Def}
We have the following properties of finite-dimensional motives:
\begin{Thm}[(Kimura, \cite{K})]\label{fin-dim} \item \begin{enumerate}[label = (\alph*)] 
\item\label{itm: curve} The motive of a smooth projective curve is finite-dimensional. 
\item\label{itm: otimes} If $M, N \in \M_{k}$ are finite-dimensional, then so are $M \oplus N$ and $M \otimes N$. Conversely, if $M \oplus N$ is finite-dimensional, then so are $M$ and $N$.
\item\label{itm: dom} If $f: M \to N$ is split-surjective and $M$ is finite-dimensional, then so is $N$. 
\item\label{itm: parity} If $M$ is finite-dimensional and the odd degree cohomology $H^{-} (M) = 0$ for some Weil cohomology $H^{*}$, then $M$ is evenly finite-dimensional (and, similarly for oddly finite-dimensionality). 
 \item\label{itm: cohom} Suppose $M$ is finite-dimensional and $\Phi \in \End_{\M_{k}} (M)$ is such that $\Phi_{*} \in \End (H^{*} (M))$ is an isomorphism. Then, $\Phi$ is an isomorphism. \end{enumerate}
\end{Thm}

As a consequence of \ref{itm: curve}, \ref{itm: otimes}, the motive of any product of smooth projective curves is finite-dimensional; from \ref{itm: dom}, so is any variety dominated by a product of curves (such as Abelian varieties). In \cite{GG} it is demonstrated that varieties of dimension $\leq 3$ have finite-dimensional motive if $CH_{0} (X)_{hom}$ is representable. (This is true, in particular, if $X$ is a rationally connected threefold.) Other than this, the following conjecture remains wide open.
\begin{Conj}[(Kimura, O'Sullivan)] Every motive $M \in \M_{k}$ is finite-dimensional.
\end{Conj}

To prove Theorem \ref{2}, we will need the next two results.

\begin{Prop}[(\cite{M} Theorem 3)]\label{M} Let $X$ be a smooth projective surface. Then, there are idempotents $\pi_{i} \in \End_{\M_{k}} (\h(X))$ satisfying the following conditions:
\begin{enumerate}[label=(\roman*)] \item $\Delta_{X}  = \pi_{0} + \pi_{1} +\pi_{2}+\pi_{3}+\pi_{4}$ \item $\pi_{i}\circ\pi_{j} =  0$ for $i \neq j$ \item The motive $\h^{i} (X) = (X, \pi_{i}, 0)$ satisfies $H^{*} (\h^{i} (X)) = H^{i} (X)$ for any Weil cohomology $H^{*}$.\end{enumerate}
\end{Prop}

\begin{Rem} It should be noted that the idempotents in Proposition \ref{M} are not unique. Some extra conditions (see loc. cit.) may be imposed on $\pi_{i}$ so that the resulting motives are unique up to isomorphism. \end{Rem}

 \indent When $X$ is an Abelian variety of dimension $g$, we have the following result of Denninger and Murre which was proved using Beauville's Fourier transform (\cite{B2}). 
\begin{Thm}[(\cite{DM} Theorem 3.1)]\label{DM} There is a unique decomposition
\begin{equation} 
\h(X) = \bigoplus_{i=0}^{2g} \h^{i} (X)
\end{equation}
where $\h^{i} (X) = (X, \pi_{i}, 0)$ with $\pi_{i}$ idempotents satisfying:
\begin{enumerate}[label=(\roman*), ref=(\roman*)]
 \item\label{itm: ortho} $\pi_{i}\circ\pi_{j} = 0$ for $i \neq j$;
 \item\label{itm: n} $\prescript{t}{}{\Gamma_{n}}\circ\pi_{i} = n^{i}\cdot\pi_{i} = \pi_{i}\circ\prescript{t}{}{\Gamma_{n}}$ for all $n \in \zit$;
 \item\label{itm: transpose} $\prescript{t}{}{\pi_{i}} = \pi_{2g-i}$.
 \end{enumerate}
\end{Thm}
\begin{Rem}  As suggested by the notation, we have $H^{*} (\h^{i} (X)) = H^{i} (X)$.
\end{Rem}
Since the motive of an Abelian variety $X$ is finite-dimensional and the cohomology of $\h^{i} (X)$ is concentrated in one degree,  Theorem \ref{fin-dim} implies that $\h^{i} (X)$ is finite-dimensional of parity $i \pmod 2$. From Theorem \ref{1} and Theorem \ref{fin-dim} \ref{itm: dom}, it follows that the motive of the Fano surface $\h(S)$ is also finite-dimensional and the summand $\h^{2} (S)$ as in Proposition \ref{M} is evenly finite-dimensional.

\begin{proof}[Proof of Theorem \ref{2}] Fix $\pi_{2, S}$ as in Proposition \ref{M}. We define $$\h^{2} (i) = \pi_{2, S} \circ \prescript{t}{}{\Gamma_{i}} \circ \pi_{2, A} \in \Hom_{\M_{k}} (\h^{2} (A), \h^{2} (S)) $$ The goal is to show that $\h^{2} (i)$ is an isomorphism of motives. This means that we need to find some $\psi \in \Hom_{\M_{k}} (\h^{2} (S), \h^{2} (A))$ for which \begin{equation} \begin{split} \psi \circ \h^{2} (i) = \pi_{2, A} \in \End_{\M_{k}} (\h^{2} (A)), \ \h^{2} (i) \circ \psi = \pi_{2, S} \in \End_{\M_{k}} (\h^{2} (S)) \end{split} \label{last} \end{equation} Since $\h^{2} (A)$ and $\h^{2} (S)$ are evenly finite-dimensional, it suffices by Theorem \ref{fin-dim} \ref{itm: cohom} to find some such $\psi$ for which these equalities hold cohomologically. Moreover, $i^{*} = \h^{2} (i)_{*} : H^{2} (A, \Q_{\ell}) \to H^{2} (S, \Q_{\ell})$ is an isomorphism for $\ell \neq \text{char } k$. Thus, if we can find some \begin{equation} \gamma \in CH^{2} (S \times A)\label{gamma} \end{equation} for which $\gamma_{*}: H^{2} (S, \Q_{\ell}) \to H^{2} (A, \Q_{\ell})$ is the inverse of $i^{*}$, we can set $\psi := \pi_{2, A}\circ\gamma\circ \pi_{2,S}$, and this gives the desired correspondence in (\ref{last}).\\
\indent From \cite{B1}, the image of the map $S \times S \to A$ defined by $(x, y) \mapsto i(x) - i(y)$ is an ample divisor $\Theta$. Moreover, we have $$\frac{1}{3!}\cdot \Theta\wedge\Theta\wedge\Theta = [S] \in H^{6} (A, \Q_{\ell} (3))$$ The Hard Lefschetz theorem then implies that $\wedge [S] : H^{2} (A, \Q_{\ell}) \to H^{8} (A, \Q_{\ell}) (3)$ is an isomorphism. Also, since the Lefschetz standard conjecture is true for Abelian varieties (\cite{KL} Proposition 4.3), it follows that there is a correspondence $\phi \in CH^{2} (A \times A)$ such that $\phi_{*}: H^{8} (A, \Q_{\ell}) \to H^{2} (A, \Q_{\ell}) (-3)$ is the inverse of $\wedge [S]$. The projection formula then shows that $i_{*}\circ i^{*} = \wedge [S]$ so that $\phi_{*}\circ i_{*}\circ i^{*}$ is the identity on $H^{2} (A, \Q_{\ell})$. Since $i^{*} : H^{2} (A, \Q_{\ell}) \to H^{2} (S, \Q_{\ell})$ is an isomorphism, it follows that $\gamma = \phi\circ \Gamma_{i}$ is the desired inverse. 

\end{proof}

\end{document}